\documentclass[11pt, a4 paper]{amsart}
\usepackage{amsmath,amsthm,amssymb,amsfonts, indentfirst}

\usepackage[plainpages=false,colorlinks,hyperindex,pdfpagemode=None,bookmarksopen,linkcolor=red,citecolor=blue,urlcolor=blue]{hyperref}
\usepackage{pdflscape}

\numberwithin{equation}{section}
\makeindex

\renewcommand{\d}{\mathrm{d}}

\numberwithin{equation}{section}

\newtheorem{Thm}{Theorem}[section]
\newtheorem{Lem}[Thm]{Lemma}
\newtheorem{Prop}[Thm]{Proposition}
\newtheorem{Cor}[Thm]{Corollary}
\newtheorem{Rem}[Thm]{Remark}
\newtheorem{Def}[Thm]{Definition}

\newtheorem{Fact}[Thm]{Fact}
\newtheorem{Nota}[Thm]{Notation}

\def\zzz{\widehat{\mathbb{Z}}^d_{\geq 2}}
\def\zhat{\widehat{\mathbb{Z}}^d}

\def\R{\mathbb{R}}

\def\N{\mathbb{N}}

\def\Z{\mathbb{Z}}
\def\T{\mathbb{T}}

\def\Coo{C^\infty}

\def\cC{\mathcal{C}}
\def\cE{\mathcal{E}}
\def\cD{\mathcal{D}}
\def\cF{\mathcal{F}}

\def\cK{\mathcal{K}}

\def\cX{\mathcal{X}}

\def\cZ{\mathcal{Z}}

\def\cM{\mathcal{M}}

\def\a{\alpha}
\def\b{\beta}

\def\G{\Gamma}

\def\d{\delta}

\def\z{\zeta}
\def\vp{\varphi}
\def\h{\theta}
\def\l{\lambda}

\def\g{\mathfrak{g}}

\def\o+{\oplus}

\def\p[#1,#2]{\phi_{#1,#2}}
\def\til[#1]{\widetilde{#1}}
\def\ba{\boldsymbol{a}}
\def\bb{\boldsymbol{b}}

\def\bx{\boldsymbol{x}}
\def\by{\boldsymbol{y}}

\def\be{\textbf{e}}

\def\bk{\textbf{k}}
\def\bn{\textbf{n}}

\def\d1{{d-1}}
\def\ahat{\widehat{\boldsymbol{a}}}
\def\ove{{-1/(d-1)}}
\def\lo[#1]{\|#1\|_{L^\infty}}
\def\l2[#1]{\|#1\|_{L^2}}
\def\z[#1]{z_{#1}}

\def\=>{\Longrightarrow}
\def\<{\langle}
\def\>{\rangle}
\def\^{\wedge}
\def\+{\dagger}
\def\~={\simeq}

\def\til[#1]{\widetilde{#1}}
\def\un[#1]{\underline{#1}}

\def\over[#1]{\overline{#1}}
\def\vec[#1]{\overrightarrow{#1}}
\def\ra[#1,#2]{{^#1}_#2}
\def\veca[#1]{\begin{pmatrix}#1\end{pmatrix}}
\def\mat[#1, #2]{\left[\begin{array}{ccccc}#1\end{array}\left|\begin{array}{c}#2\end{array}\right.\right]}
\def\xto[#1]{\xrightarrow{#1}}
\def\dd[#1,#2]{\frac{#1}{#2}}
\def\del[#1,#2]{\frac{\partial #1}{\partial #2}}
\def\Facts[#1]{\begin{Fact}\mbox{}\begin{itemize}#1\end{itemize}\end{Fact}}
\def\Notation[#1]{\begin{Nota}\mbox{}\begin{itemize}#1\end{itemize}\end{Nota}}
\def\no[#1]{\| #1\|}
\def\lip[#1]{\|#1\|_{Lip}}
\def\tt{{\T^\d1}}
\DeclareFontFamily{OT1}{rsfs}{}
\DeclareFontShape{OT1}{rsfs}{n}{it}{<-> rsfs10}{}
\DeclareMathAlphabet{\mathscr}{OT1}{rsfs}{n}{it}

\begin{document}
\title{EFFECTIVE LIMIT DISTRIBUTION OF THE FROBENIUS NUMBERS}



\author{HAN LI }

\begin{abstract} The Frobenius number $F(\ba)$ of a lattice point $\ba$ in $\R^d$ with positive coprime coordinates, is the largest integer which can $not$ be expressed as a non-negative integer linear combination of the coordinates of $\ba$. Marklof in \cite{M} proved the existence of the limit distribution of the Frobenius numbers, when $\ba$ is taken to be random in an enlarging domain in $\R^d$. We will show that if the domain has piecewise smooth boundary, the error term for the convergence of the distribution function is at most a polynomial in the enlarging factor.
\end{abstract}
\maketitle

\section{\bf Introduction}

Let $\widehat{\Z }^d=\{\ba=(a_1,...a_d)\in\Z^d :  $gcd$ (a_1, ..., a_d)=1 \}$ be the set of primitive lattice points, and $\zzz$ be the subset of $\widehat{\Z }^d$ with coordinates $a_i\geq 2$. For any $\boldsymbol a\in\zzz$, there exists a largest natural number $F(\ba)$, that is not representable as a linear non-negative integer combination of the coordinates of $\ba$. The number $F(\ba)$ is called the Frobenius number of vector $\ba$, i.e. $$F(\ba)=\max \left\{\N \setminus  \{\boldsymbol k\cdot \ba: \boldsymbol k=(k_1,...k_d)\in\Z^d, k_i\geq0 \}\right\}.$$ 

The Frobenius number problem is also known as ``the Coin Exchange Problem". It has been studied extensively in the past 50 years using various techniques, such as combinatorics, probabilistic method, geometry of numbers, and more recently homogeneous dynamics, etc. When $d=2$, Sylvester showed in 1884 that $F(\ba)=a_1a_2-a_1-a_2$, and no explicit formula is known when $d\geq 3$. However, several upper bounds of the Frobenius numbers were obtained by the 1980's. Let $\ba\in\zzz$ with $a_1\leq a_2\leq \cdots \leq a_d$, the estimates include the work by Erd\"os and Graham \cite{EG}
\begin{equation}\label{2}
F(\ba)\leq 2a_{d-1}\left [ \dd[a_d, d]\right ]-a_d,
\end{equation}
and the work by Selmer \cite{SE}
\begin{equation}\label{1}
F(\ba)\leq 2a_d\left [\dd[a_1, d]\right ]-a_1.
\end{equation}

There are also results on the limit distribution of Frobenius numbers from different perspectives. In dimension $d=3$  using continued fractions, Bourgain-Sinai (\cite{BS}) showed that for ensembles $\Omega_N=\{a\in\zzz: a_i\leq N\}$, the limit distribution of $\dd[F(\ba), N^{3/2}]$ exists. Marklof  in \cite{M} generalized their result to higher dimensions. Before stating his results, let us first recall the notion of the covering radius. A lattice $L$ in $\R^\d1$ is a discrete additive subgroup of $\R^\d1$ with finite covolume ${\rm det}(L)$, which equals the volume of the fundamental domain of the $L$ action on $\R^\d1$. The covering radius (denoted by $Q_0$) of a lattice $L\subset\R^\d1$ is given by
\begin{equation}\label{q0}
Q_0(L)=\inf \left\{ t\in \R^+: t\Delta_{d-1} + L= \R^{d-1} \right\},
\end{equation}  where $\Delta_\d1= \left\{\boldsymbol x\in \R^\d1 :  x_i \geq 0,\ \sum^\d1_{i=1} x_i \leq 1  \right\}$ is the standard simplex. 

Lattices of covolume 1 are called unimodular. Let $G_0={\mathrm{SL}}_{d-1}(\R)$, $\Gamma_0=\rm{SL}_{d-1}(\Z)$, then $\Omega_0=G_0/\Gamma_0$ can be identified with the space of unimodular lattices in $\R^\d1$ ($g\Gamma_0\leftrightarrow g\Z^{d-1}$). Let us fix a right invariant Riemannian metric on $G_0$, giving rise to a metric and a left $G_0-$invariant probability measure $\bar{\mu}_0$ on $\Omega_0$. Let us recall
\begin{Thm}[\cite{M}]\label{Ma} Let $d\geq 3$ and $E_R=\{ L\in \Omega_{0}: Q_{0}(L) \leq R\}$. Then
\begin{itemize}
\item [(i)] for any bounded set $\cD \subset \R^d_{>0}$ with boundary of Lebesgue measure zero, and any $R \geq 0$, 
\begin{equation}\label{erro}
\lim _{T\rightarrow \infty} \dd[1, T^d] \# \left\{\ba\in \widehat{\Z}^d_{\geq 2} \cap T \cD : \dd[F(\ba), (a_1\cdots a_d)^{1/(d-1)}] \leq R \right\}  = \dd[\rm{vol}(\cD), \zeta (d)] \bar{\mu}_0(E_R).
\end{equation}
\item[(ii)] $Q_0$ is a continuous function on $\Omega_0$.
\item[(iii)] $\bar{\mu}_0(E_R)$ is continuous in $R$, i.e. $\bar{\mu}_0(\{L\in\Omega_0: Q_0(L)=R\})=0$ for any $R>0$.
\end{itemize}
\end{Thm}


We now briefly explain the existence of the limit distribution based on our private correspondence with Marklof. This is what we are going to follow in this paper which is more suitable for the purpose of ``effectivization'', and is slightly different from \cite{M}. The method is based on homogeneous dynamics, which is also combined with the geometry of numbers. Here are the main ideas: Aliev and Gruber showed in \cite{AG} that for any $\ba\in\zzz$, one associates a $d-1$ dimensional unimodular lattice $L_{\ba}\in\Omega_0$ (see Theorem \ref{ink}) with
\begin{equation}\label{sug}
\dd[F(\ba)+\sum^d_{i=1}a_i, (a_1\cdots a_d)^{1/(d-1)}]=Q_0(L_{\ba}).  
\end{equation}
This is essentially due to a geometric interpretation of the Frobenius numbers found by Kannan (\cite{KA}). 
Let  $\cD\subset\R^d_{>0}$ be a bounded connected non-empty open subset with boundary of Lebesgue measure zero. Notice that we have
\begin{equation}\label{asy}
|T\cD\cap\zhat|\sim\dd[T^d\rm{vol}(\cD), \zeta(d)].
\end{equation} 
As we will see in Section 4, the set of lattices $\{L_{\ba}: \ba\in T\cD\cap\zhat\} $ appearing in (\ref{sug}) becomes equidistributed in $\Omega_0$ as $T\rightarrow\infty$, i.e, for any bounded continuous function $\phi$ on $\Omega_0$, we have
\begin{equation}\label{con}
\lim_{T\rightarrow\infty}\dd[1, T^d]\sum_{\ba\in T\mathcal{D}\cap\zhat}\phi(L_a)=\dd[\rm{vol}(\cD), \zeta(d)]\int_{\Omega_0}\phi d\bar{\mu}_0.\end{equation}
Since $E_R=\{ L\in \Omega_{0}: Q_{0}(L) \leq R\}$ has boundary measure zero, Theorem \ref{Ma} can be deduced by applying $\chi_{E_R}$ to (\ref{con}) in the place of $\phi$.

Theorem \ref{Ma} also implies that for large $R$ and $T$, the probability that a random lattice point $\ba\in T\cD\cap\zhat$ satisfies $F(\ba)<R(a_1\cdots a_d)^{1/(d-1)}$ is greater than 99\%.  This gives a somewhat better estimate compared with (\ref{2}) and (\ref{1}), for most $\ba\in T\cD\cap\zhat$.

The aim of this paper is to estimate the decay of the function $\Psi(R)=1-\bar{\mu}_0(E_R)$ and the error term of (\ref{erro}).
\begin{Thm}\label{tail} There exists a constant $C>0$ dependent only on $d$, such that for any $R>0$ we have $\Psi(R)< CR^{-(\d1)}$. \end{Thm}

Theorem \ref{tail} improves the exponent compared with Theorem 1 of \cite{AHH}. After this paper was completed, Marklof, in an unpublished work \cite{M2} proved that there exists a constant $c_d>0$, so that $\Psi(R)>c_d R^{-(\d1)}.$ Therefore our bound is actually sharp. An asymptotic formula for $\Psi(R)$ has recently been obtained by Str\"ombergsson in \cite{AN}.

\begin{Def}\label{smooth}
We say that a subset of a manifold $M$ has ``thin boundary'', if the boundary is contained in a union of finitely many smooth connected submanifolds of $M$ whose dimensions are strictly less than the dimension of $M$.
\end{Def}
 
\begin{Thm}\label{main}
There exists $\kappa>0$ dependent only on $d$ satisfying the following property. For any $R> 0$, and any non-empty connected open subset $\cD \subseteq \{\bx\in\R^d: 0<x_i<1\}$, which has thin boundary as a subset of the manifold $\R^d$,  there exist constants $C_R$, $C_\cD>0$ such that for every $T\geq 1$
\begin{equation}\label{prob}
\left |  \dd[1, T^d] \# \left\{ \ba\in T \cD\cap \zzz : \dd[F(\ba)+\sum^d_{i=1}a_i, (a_1\cdots a_d)^{1/(d-1)}] \leq R \right\}  -\dd[\rm{vol}(\cD), \zeta(d)]\bar{\mu}_{0}(E_R)\right |< \dd[C_R C_\cD, T^\kappa].  \nonumber
\end{equation}
\end{Thm}

When $d=3$, more explicit calculation was done by Ustinov in \cite{U10} for $\cD=\{\bx\in\R^3: 0<x_i<1\}$, and Theorem \ref{main} is consistent with his result. 



\textbf{Organization of the paper}  In section 2 we will use the geometry of numbers to prove Theorem \ref{tail}. We will also give an explicit description for the $L_{\ba}$ appearing in formula (\ref{sug}). Section 3 and 4 are devoted to proving effective equidistribution of both expanding horospheres, and a Farey sequence on a specified closed horosphere, under the translation of a one sided diagonal flow. We will give an error term estimate of (\ref{con}) for non-negative compactly supported $C^1$ test functions (Theorem \ref{formain}). Theorem \ref{main} will be then proved when we show that the boundary of $E_R=\{L\in\Omega_0: Q_0(L)\leq R\}$ is thin. We will borrow many ideas from \cite{M} in Section 4 to formulate a series of equidistribution results which lead to Theorem \ref{formain}. 

\textbf{Notation} Throughout the paper we assume that the dimension $d\geq 3$, and always work with column vectors in Euclidean spaces. We use $\ll$ to represent inequalities in which the implicit constants depend on the underlying Lie groups or Euclidean spaces. In a metric space $X$, $B_X(x, r)$ stands for the open ball the radius $r$ centered at $x$. On a Lie group $G$, $B_G(r)=B_G(e, r)$ with a specified metric on $G$; in $\R^n$, $B(r)=B_{\R^n}(0, r)$ with Euclidean norm. The exponents $\alpha_1, \alpha_2, \cdots$ in section 3 and 4 depend only on the dimension $d$.

\textbf{Acknowledgments} I would like to thank J. Marklof and J. Athreya for their discussions on the concepts and ideas related to this project, and for their comments on the early draft of this paper. I am grateful to Marklof for pointing out a problem in the original formulation of Theorem \ref{main}. I am deeply indebted to my thesis advisor Prof. Gregory Margulis for suggesting me this project, for his invaluable discussions and constant encouragement. I also want to thank the anonymous referee for many helpful suggestions and corrections.

\section{\bf Covering Radius and the Frobenius Numbers}
We call a subset $K$ of $\R^\d1$ a $convex\ body$ if $K$ is a compact convex set with non-empty interior. A convex body is called centrally symmetric if it is symmetric with respect to the origin. For a centrally symmetric convex body $K$, its polar $K^*$ is also a centrally symmetric convex body defined by $K^*=\{\bx\in\R^\d1: \bx\cdot\by\leq 1, {\rm{for\ any}}\ \by\in K\}.$

We now recall the notion of $dual\ lattice$. Let $L=A\Z^\d1\in\Omega_0$ where $A\in G_0$, and let $A^*$ be the inverse transpose of $A$. We call the lattice $L^*=A^*\Z^\d1$ the dual lattice of $L$. One readily verifies that the definition of $L^*$ is independent of the choice $A$, and moreover the map $L\rightarrow L^*$ is a diffeomorphism of $\Omega_0$ which preserves $\bar{\mu}_0$. 

The covering radius $Q(K, L)$ of $K\subseteq\R^\d1$ with respect to a lattice $L$ in $\R^\d1$ is defined by $$Q(K, L):= \inf \left\{ t\in \R^+: tK + L= \R^{d-1} \right\}.$$ Clearly the function $Q_0$ defined in (\ref{q0}) satisfies $Q_0(L)
 =Q(\Delta_\d1 ,L)$ for any lattice $L$ in $\R^\d1$. (We will abbreviate $\Delta_\d1$ as $\Delta$ in what follows.)



The covering radius is related to the $Minkowski's\ successive\ minima$. Let $K\subseteq\R^\d1$ be a centrally symmetric convex body, and $L$ be a lattice in $\R^\d1$, the $i$-th minimum ($1\leq i\leq d-1$) of $K$ with respect to $L$ is defined by 
$$\lambda_i(K, L):=\min\left\{ t\in\R^+: \dim({\rm{span}}(tK\cap L)\geq i)    \right\}.$$
Clearly $0<\lambda_1(K, L)\leq\lambda_2(K, L)\cdots\leq\lambda_{d-1}(K, L).$

\begin{Lem}[Minkowski's Second Theorem]Let $K\subseteq\R^\d1$ be a centrally symmetric convex body and $L$ be a lattice in $\R^\d1$. Then 
$$\dd[2^\d1, (d-1)!]\leq \dd[{\mathrm{vol}}(K),{ \mathrm{det}}(L)]\prod _{i=1}^{d-1}\lambda_i(K, L)\leq 2^\d1.$$

\end{Lem}


\begin{Lem}[Kannan-Lov\' asz, 2.4, 2.8 \cite{KL}] Let $K$ be a convex body and $L$ be a lattice in $\R^\d1$, and set $K-K=\{\bk_1-\bk_2: \bk_1, \bk_2\in K\}$. Then
\begin{itemize}
\item[(i)]$\lambda_\d1(K-K, L)\leq Q(K, L)\leq \sum_{i=1}^\d1\lambda_i(K-K, L),$
\item[(ii)]There exists a constant $C_d>0$, so that $\lambda_\d1(K-K, L)\lambda_1((K-K)^*, L^*)<C_d$.
\end{itemize}

\end{Lem}



\begin{Lem}\label{compact}
The function $Q_0$ defined in (\ref{q0}) is proper on $\Omega_0$, i.e. $E_R=\{ L\in \Omega_{0}: Q_{0}(L) \leq R\}$ is a compact subset of $\Omega_0$ for any $R\geq 0$.
\end{Lem}
\begin{proof}
By Lemma 2.1 and (i) of 2.2, $\lambda_1(\Delta-\Delta, L)$ is positively bounded below for $L\in E_R$. Mahler's Criterion implies that $E_R$ is relatively compact in $\Omega_0$. Since $Q_0$ is continuous (see (ii) of Theorem \ref{Ma}), $E_R$ is compact. 
\end{proof}

\begin{Lem}{\rm (Lemma 4.1 of \cite{AM})} For any centrally symmetric convex body $K$ in $\R^\d1$, there exists a constant $C_K>0$ so that for any $r>0$,
$$\bar{\mu}_0(\{L \in\Omega_0: \lambda_1(K, L)<r\})< C_Kr^{\d1}.$$
\end{Lem}

\begin{flushleft}{\bf Proof of Theorem \ref{tail}:}\end{flushleft}By Lemma 2.2, for any $R>0$, 
\begin{eqnarray}
\{L\in\Omega_0: Q_0(L)>R\}&\subseteq&\{L\in\Omega_0: \lambda_\d1(\Delta-\Delta, L)>\dd[R, \d1]\} \nonumber\\
&\subseteq&\{L\in\Omega_0: \lambda_1((\Delta-\Delta)^*, L^*)<\dd[(d-1)C_d, R]\} \nonumber
\end{eqnarray}
Since the map $L\rightarrow L^*$ preserves $\bar{\mu}_0$, by Lemma 2.4 $$\Psi(R)=\bar{\mu}_0(\{L \in\Omega_0: Q_0( L)>R\})\ll R^{-(\d1)}.$$ This completes the proof of Theorem \ref{tail}.               $\Box$ 

For $T>0$, $\boldsymbol{x}\in\R^\d1$ and $\boldsymbol{y}\in\R^\d1$ with each coordinate $y_i\neq 0$, we define
$$ D(T)=
\begin{pmatrix}
T^{-1/(d-1)}1_{d-1} & 0\\
0 & T
\end{pmatrix}    \qquad
n(\boldsymbol{x})=\begin{pmatrix}
1_{d-1} & 0\\
\boldsymbol{x}^t & 1
\end{pmatrix}\qquad
m(\boldsymbol{y})=\begin{pmatrix}
m'(\boldsymbol{y})& 0\\
0 & 1
\end{pmatrix}$$where $m'(\boldsymbol{y})=(y_1\cdots y_\d1)^{-\dd[1, \d1]}diag(y_1,\cdots, y_\d1)$. Clearly $D(T), n(\boldsymbol{x})$, $m(\boldsymbol{y})\in\rm{ SL}_d(\R)$.

\begin{Def}For every $\ba\in\Z^d$ with $a_d\neq 0$, we associate a vector $\ahat\in\R^{d-1}$ by $$\ahat=(\dd[a_1, a_d],\cdots, \dd[a_\d1,a_d])^t.$$
\end{Def}
For any $\ba\in\zzz$ let $M_{\ba}$ be the lattice $M_{\ba}:=\{\boldsymbol{b}\in\Z^\d1: \ba\cdot\bb \equiv 0\pmod{a_d} \}.$ Since $\boldsymbol{a}$ is primitive, $M_{\ba}$ has determinant $a_d$. Aliev-Gruber, based on the work of Kannan \cite{KA}, has shown in \cite{AG} that the Frobenius number $F(\ba)$ satisfies 
$$ \dd[F(\ba)+\sum^d_{i=1}a_i, (a_1\cdots a_d)^{1/(d-1)}]=Q_0(m'(\widehat{\boldsymbol{a}})(a_d^{\ove}M_{\ba})).$$
This enables us to present an explicit description of the lattice $L_{\ba}$ in formula (\ref{sug}):
\begin{Thm} \label{ink}For any $\ba\in\zzz$ the Frobenius number $F(\ba)$ satisfies
\begin{equation} \dd[F(\ba)+\sum^d_{i=1}a_i, (a_1\cdots a_d)^{1/(d-1)}]=Q_0(m(\widehat{\boldsymbol{a}})D(a_d)n(\ahat)\Z^d\cap\boldsymbol{e}_d^\perp), \end{equation}
where $\boldsymbol{e}_d=(0,\cdots,0, 1)^t$, and we identify $\R^d\cap\boldsymbol{e}_d^{\perp}$ with $\R^{d-1}$ in the obvious way. In other words, $L_{\ba}=m(\widehat{\boldsymbol{a}})D(a_d)n(\ahat)\Z^d\cap\boldsymbol{e}_d^\perp$.
\end{Thm}
\begin{proof}Note that for any $\by=(y_1,\cdots, y_d)^t\in\R^d$ with $\by\cdot\ba=0$, we have $n(\ahat)\by=(y_1,\cdots,y_\d1, 0)^t$. Therefore $ M_{\ba}=(n(\ahat)\Z^d)\cap\boldsymbol{e}_d^\perp$. The conclusion follows immediately from the above lemma.
\end{proof}

\section{\bf Translations of Horospheres and Effective Equidistribution} 
Let $G={\rm SL}_d(\R), \Gamma={\rm SL}_d(\Z)$, $G_0={\rm SL}_\d1(\R)$. We will identify $G_0$ with the image of the embedding $A\rightarrow\begin{pmatrix}
A & 0\\
0 & 1
\end{pmatrix}$. 
We denote by $F=\{D(s): s>0\}$ the subgroup of $G$, and set $F^+=\{D(s): s>1\}$. For the subgroups of $G$ $$H=\{n(\bx): \bx\in\R^\d1\},\ \ \
H^- = \{n(\bx)^t:  \bx\in\R^\d1\}, \ \ \ H_0=\left\{\begin{pmatrix}
A & 0\\
0 & c
\end{pmatrix}: {\rm det}(A)c=1\right\},$$
their Lie algebras are invariant subspaces of the adjoint action of $F$ on $\g={\mathrm {Lie}}(G)$.


We identity the Lie algebra $\g$ with the space of $d\times d$ traceless matrices, and define an inner product on $\g$ by setting $\langle X, Y\rangle={\mathrm {tr}}(X^tY)$ ($X, Y\in\g$). This gives rise to a right invariant Riemannian metric on $G$, and hence a metric $d_S$ on any closed subgroup $S$ of $G$. We have an orthogonal decomposition of $\g$ into linear subspaces $\g={\mathrm {Lie}}(H)+{\mathrm {Lie}}(H^-)+{\mathrm {Lie}}(G_0)+{\mathrm {Lie}}(F)$. We thus fix an orthonormal basis of $\g$ coming from a basis of those subspaces
\begin{equation}\label{newbasis}
\cX=\left\{X_i: i=1, 2, \cdots, d^2-1   \right\}.
\end{equation}
We define for every $s>0$ a map $\phi_s: G\rightarrow G$ by  $$\phi_s(g)=D(s)gD(s^{-1}).$$ The restriction of $\phi_s(s>1)$ on $H'=H_0H^-$ is thus contracting in the sense that
$\phi_s(B_{H'}(r))\subseteq B_{H'}(r)$ for any $r>0$.
The group $H$ is called the expanding horospherical subgroup with respect to $F^+$ in the sense that
$H=\{g\in G: D(s^{-1})gD(s)\rightarrow 0, {\rm{as}}\ s\rightarrow +\infty\}.$

Let $\Omega=G/\Gamma$ with the metric $d_\Omega$ coming from $G$. Every $H-$orbit in $\Omega$ is called an expanding horosphere (with respect to $F^+$).
We specify a closed horosphere $$Y=\{h\Gamma: h\in H\}=\{n(\bx)\G: \bx\in\T^\d1\}\subseteq\Omega.$$By $\mu$ and $\nu$ we donte the left Haar measures on $G$ and $H$ respectively, with the induced measures on $\Omega$ and $Y$ satisfying $\bar{\mu}(\Omega)=1$ and $\bar{\nu}(Y)=1$ (which means $\nu$ and $\bar{\nu}$ correspond to the Lebesgue measures on $\R^\d1$ and $\T^\d1$ respectively). We choose a left Haar measure $\nu'$ on $H'$ so that $\mu$ is locally the product of $\nu$ and $\nu'$. This means, in view of Theorem 8.32 in \cite{BO}, for any $f\in L^1(G)$:
\begin{equation}\label{int}
\int_{H'H}f(g)d\mu=\int_{H'\times H}f(h'h)d\nu'(h')d\nu(h).
\end{equation}



\bigskip

\begin{flushleft}\textbf{Decay of Matrix Coefficients and Its Consequences} \end{flushleft}

Let $\rho$ be a (strongly continuous) unitary representation of $G$ on a Hilbert space $\mathcal{H}$. We say a vector $v\in\mathcal{H}$ is Lipschitz if (the metric $d$ below refers to the fixed metric on $G$)
$$Lip(v):=\sup\left\{\dd[\no[\rho(g)v-v], d(e, g)]: g\neq e \right\}<\infty.$$
The theory of ``matrix coefficients decay'' has been developed by many people, such as Harish-Chandra, Moore, Howe, Cowling, Katok-Spatzier, etc. Based on the previous works, Kleinbock-Margulis proved a quantitative decay of matrix coefficients for Lipschitz vectors. For our purpose we only need the following theorem which follows from Theorem A.4 of \cite{KM1}, combined with Kazhdan's property (T) for the groups $G={\rm SL}_n(\R)\ (n\geq 3)$.
 \begin{Thm}\label{KM} There exists $\alpha_1>0$ so that for any unitary representation ($\rho, \mathcal{H}$) of $G$ without G-invariant vectors, any Lipschitz vectors $v, w\in\mathcal{H}$ and every $s>1$, we have that 
$$|\langle\rho(D(s))v, w\rangle|\ll s^{-\alpha_1}(Lip(v)+\no[v])(Lip(w)+\no[w]).$$
\end{Thm}


\begin{Def}We say that a function $\psi$ on a metric (``dist'') space $X$ is Lipschitz if 
$$\|\psi\|_{Lip}:= \sup\left\{\dd[|\psi(x)-\psi(y)|, dist(x, y)]: x, y\in X, x\neq y\right\} <\infty,$$
The space of Lipschitz functions on $X$ is denoted by $Lip(X)$.
\end{Def}
\begin{Rem}\label{equal}If $X$ is a Riemannian manifold with distance coming from the Riemannian metric and $\psi$ a real-valued smooth function on $X$. Then $\no[\psi]_{Lip}=\sup\{\no[d\psi_x]: x\in X\},$
where $d\psi_x$ is the tangent map of $\psi$ at $x$, and its norm comes from the Riemannian metric on $X$.
\end{Rem}

We specify a metric on $H\times\Omega$ by setting $d((h_1, z_1), (h_2, z_2)):= d_H(h_1, h_2)+d_\Omega(z_1, z_2)$, where $h_1, h_2\in H$ and $z_1, z_2\in\Omega$. In what follows, the metrics on the product spaces are all defined in the same way. We now fix a subset
$$U=\{n(\bx): \bx\in (-2,2)^{d-1}\}\subset H.$$ 
Consider the action of $G$ on $H\times\Omega$ by $g.(h, z)=(h, gz)$ and the associated unitary representation of $G$ on $L^2(H\times\Omega)$. In this case any Lipschitz function on $H\times\Omega$ which is supported on $U\times\Omega$, is  a Lipschitz vector in $L^2(H\times\Omega)$, and moreover $Lip(\psi)\ll_U\lip[\psi]$. Notice that
$\mathcal{H}_0=\left\{\vp: \vp(h, z)=f(h),\ {\rm{for\ some}}\ f\in L^2(H) \right\}$ is the linear subspace of the $G$-stable vectors in $L^2(H\times\Omega)$. Considering the representation of $G$ on $\mathcal{H}_0^\perp$, we get that

\begin{Cor}\label{lip}
Let $\mathcal{P}: L^2(H\times\Omega)\rightarrow\mathcal{H}_0$ be the orthogonal projection. Then for any $s>1$ and functions $\vp, \psi\in Lip(H\times\Omega)\cap L^2(H\times\Omega)$ which are supported on $U\times\Omega$, we have that 
\begin{equation}\label{lipdecay}
|\langle D(s)\vp, \psi\rangle-\langle\mathcal{P}\vp, \mathcal{P}\psi\rangle|\ll(\|\vp\|_{L^2}+\|\vp\|_{Lip})(\l2[\psi]+\lip[\psi])s^{-\alpha_1}.
\end{equation} 

\end{Cor}
\bigskip
\begin{flushleft} \textbf{Effective Equidistribution and $F^+-$translations} \end{flushleft}
\begin{Def}\label{haowan} Let $M$ be a smooth manifold on which $G$ acts by diffeomorphisms. We define for every $X\in\g$ a vector field on $M$ and its corresponding tangent vector $\partial_mX$ at $m\in M$ by
$$\partial X(f)(m):=\lim_{t\rightarrow 0}\dd[f(\exp(tX)m)-f(m), t];\ \partial_mX(f)=\partial X(f)(m), \quad \forall\ f\in \Coo(M), m\in M.$$
\end{Def}
\begin{Rem}\label{boundedby} In what follows we will define $C^1$ norms for smooth functions on various Riemannian manifolds via Definition \ref{haowan}. We will always ensure that $\no[f]_{C^1}$ and $\no[f]_{\infty}+\no[f]_{Lip}$ bound each other by multiple constants depending on the manifold.  
\end{Rem}
We are going to present a quantitative equidistribution result of the $F^+-$translations of the $H-$orbit $\{(h, hx): h\in H\}$ (where $x\in\Omega$) in $H\times\Omega$. The method in our approach is by no means new. The proof here is a modification of the proofs for the equidistribution of the $F^+-$translations of $\{hx: h\in H\}$ in $\Omega$ (c.f. \cite{KM1} and \cite{KM2}). The technique is sometimes known as ``equidistribution via mixing'', which originated in Margulis' thesis. In contrast to \cite{KM1} and \cite{KM2}, the additional variable on the horosperical subgroup $H$ deserves special care when we do the ``thickening''. First we recall a well-known 
\begin{Lem}\label{sob}
For any $0<r<1$, there exists a nonnegative function $\theta\in\Coo(\R^n)$ supported in $B(r)$, such that $\int_{\R^n}\theta=1$, $\l2[\h]\ll r^{-n/2}$ and $\no[\theta]_{C^1}\ll r^{-n-1}$.
\end{Lem}

\begin{Thm}\label{forequ} Let $f\in C^1(H)$, $0<r<1$ be such that $B_H(r){\rm{supp}}(f)\subset U$, and let $x\in\Omega$ be such that $\pi_x: G\rightarrow \Omega, \pi_x(g)=gx$ is injective on $B_{G}(r)supp(f)$. Then for any $s>1$ and $\vp\in C^1(H\times\Omega)$ with $supp(\vp)\subseteq U\times\Omega$, we have that (the $\alpha_1$ below is as in Theorem \ref{KM})
\begin{eqnarray}\label{forforequ}
&\ &\Big |\int_H f(h)\vp(h, D(s)hx) d\nu(h)-\int_{H\times\Omega}f(h)\vp(h, z)d\nu(h)d\bar{\mu}(z)\Big | \nonumber\\
&\ll&\lip[\vp]\cdot r\cdot\no[f]_{L^1} +r^{-d^2}\no[f]_{C^1}\no[\vp]_{C^1} s^{-\alpha_1}.
\end{eqnarray}
The $C^1$ norms here for smooth functions on various manifolds are taken to be $\no[f]_{\infty}+\no[f]_{Lip}$.
\end{Thm}

\begin{proof}
Replace $\vp$ by $\vp(h, z)-\int_\Omega\vp(h, z)d\bar{\mu}(z)$ if necessary, we may assume $\int_\Omega \vp(h, z)d\bar{\mu}(z)=0$ for every $h\in H$. 
We choose nonnegative functions $\h'\in\Coo(H')$, $\h_1\in\Coo(H)$ supported in $B_{H'}(r)$ and $B_H(r)$ respectively by Lemma \ref{sob}. We define a function $\psi$ on $H\times G$ by setting
$$\psi(h_1h_2, h'h_2)=\h'(h')\h_1(h_1)f(h_2),\quad \forall\ h_1, h_2\in H, h'\in H',$$ and setting $\psi(h,g)=0$ outside the open subset $H\times H'H$ of $H\times G$. 
Since $\int_{H'}\h'=\int_H\h_1=1$, Theorem 8.32 of \cite{BO} and formula (\ref{int}) implies
$$\int_H f(h)\vp(h, D(s)hx) d\nu(h) = \int_{H\times H'\times H}\psi(h_1, h'h_2)\vp(h_2, D(s)h_2x)d\nu(h_1)d\nu'(h')d\nu(h_2).$$
Let us define a function $\psi_x$ on $H\times\Omega$ by setting $\psi_x(h, gx)=\psi(h, g)$ for $(h,g)\in H\times(B_G(r)supp(f))$, and $\psi_x(h, z)=0$ outside the open subset $H\times B_G(r)supp(f)\Gamma\subseteq H\times\Omega$. The definition makes sense because of the injectivity assumption. It is easy to check that $\psi_x\in C^1(H\times\Omega)$, and $supp(\psi_x)\subseteq U\times\Omega$. As $\phi_s$'s (page 5) are contractions on $H'$ we have \begin{eqnarray}
 &\ &\Big |\int_H f(h)\vp(h, D(s)hx) d\nu(h)-<D(s)\psi_x, \vp>\Big | \nonumber\\
 &=& \Big |\int_{H\times H'\times H}\psi(h_1, h'h_2)\Big(\vp(h_2, D(s)h_2x)-\vp(h_2, D(s)h'h_2x)\Big)d\nu(h_1)d\nu'(h')d\nu(h_2)\Big | \nonumber\\
  &\ll& \lip[\vp]\cdot r\cdot \no[\psi]_{L^1} =\lip[\vp]\cdot r\cdot \no[f]_{L^1}. \nonumber
    \end{eqnarray}
On the other hand
$$\lip[\psi_x]\ll \no[\h'(h')\h_1(h_1)f(h_2)]_{C^1}\ll r^{-d^2}\no[f]_{C^1},\qquad \l2[\psi_x]\ll r^{-d^2/2}\l2[f].$$
Let $\mathcal{P}: L^2(H\times\Omega)\rightarrow\mathcal{H}_0$ be the orthogonal projection as in Corollary \ref{lip}. So $\mathcal{P}(\vp)=\int_\Omega\vp(h, z)d\bar{\mu}(z)=0$ by our assumption. Due to (\ref{lipdecay})
$$|\langle D(s)\psi_x, \vp\rangle|\ll r^{-d^2}\no[f]_{C^1}(\lip[\vp]+\l2[\vp])s^{-\alpha_1}\ll  r^{-d^2}\no[f]_{C^1}\no[\vp]_{C^1}s^{-\alpha_1}.$$
The theorem follows immediately.
\end{proof}

The following theorem concerns the equidistribution of $F^+-$translations of the Lebesgue measure on $\{(\bx, n(\bx)\G): \bx\in I^\d1 \}$ where $I=(0, 1)$. The result without the additional variable on $I^\d1$ is the classical equidistribution of large closed horospheres. The reason that we also consider a variable on $I^\d1$ here is that the matrix $m(\bx)$, which is related to Frobenius numbers via Theorem \ref{ink}, is defined for $\bx\in I^\d1$. This will get involved in Theorem \ref{twist}.

\begin{Thm}\label{horo}There exists a constant $\a_2>0$ such that any $\phi\in C^1(I^\d1\times\Omega)$ and $T>1$,
\begin{equation}
\Big |\int_{I^\d1} \phi(\bx, D(T)n(\bx)\Gamma) d\bx-\int_{I^\d1 \times\Omega}\phi(\bx, z) d\bx d\bar{\mu}(z) \Big | \ll \no[\phi]_{C^1} T^{-\a_2}.
\end{equation}
Here $d\bx$ is the Lebesgue measure, and $\no[\phi]_{C^1}:=\lo[f]+\max\{\|\dd[\partial,\partial x_i]f\|_{L^{\infty}}, \lo[\partial X(f)] \}$ where $\dd[\partial, \partial x_i]$ are the standard Euclidean vector fields for the  $I^{d-1}$ factor, and the $X\in\cX$ are vector fields for the $\Omega$ factor. (See (\ref{newbasis}) and Definition \ref{haowan}) 
\end{Thm}
\begin{proof}
In order to apply Theorem \ref{forequ} and get an error term estimate, we need to approximate both $\chi_{I^\d1}$ and $\phi$ by $C^1$ functions on $\R^\d1$ and $\R^\d1\times\Omega$ respectively. 

Let's fix a partition $\{E_i: 1\leq i\leq N\}$ of $I^\d1$ with the interior of each $E_i$ being an open cube, and choose $r_0>0$ so that for each $i$ we have that $B_H(r_0)\{n(\bx): \bx\in E_i\}\subseteq U$, and the restriction of $\pi: G\rightarrow\Omega,\ \pi(g)=g\Gamma$ to $\left\{gn(\bx): g\in B_G(r_0), \bx\in E_i\right\}$ is injective. Hence for every $0<r<r_0$ and $1\leq i\leq N$, there exists a function $p_i\in C^1(\R^\d1)$ supported in $E_i$, 
$$0\leq p_i\leq \chi_{E_i},\quad \mathrm{vol}(\{\bx\in E_i: p_i(\bx)\neq 1\})\ll r,\quad \no[p_i]_{C^1}\ll r^{-1}.$$

 There is also a function $p\in C^1(\R^\d1)$ supported in $I^\d1$ with 
  $$0\leq p\leq\chi_{I^\d1},\quad \mathrm{vol}(\{\bx\in I^\d1: p(\bx)\neq 1\})\ll r^{1/2},\quad \no[p]_{C^1}\ll r^{-1/2}.$$ 
 Letting $\tilde{\phi}(\bx, z)=p(\bx)\phi(\bx, z)\in C^1(\R^\d1\times\Omega)$, we then have 
$$\Big|\int_{\R^\d1} \chi_{E_i} (\bx)\phi(\bx, D(T)n(\bx)\Gamma)d\bx-\int_{\R^\d1} p_i(\bx)\tilde{\phi}(\bx, D(T)n(\bx)\Gamma)d\bx\Big|\ll C\lo[\phi]r^{1/2},$$
$$\Big |\int_{\R^\d1\times\Omega} \Big(\chi_{E_i} (\bx)\phi(\bx, z)-p_i(\bx)\tilde{\phi}(\bx, z)\Big) d\bx d\bar{\mu}(z)\Big |\ll C\lo[\phi]r^{1/2}.$$
As $p_i(\bx)$ and $\tilde{\phi}(\bx, z)$ satisfy the assumptions of Theorem \ref{forequ}, we have that
\begin{eqnarray}&\ &\Big |\int_{\R^\d1} p_i(\bx)\tilde{\phi}(\bx, D(T)n(\bx)\Gamma) d\bx-\int_{\R^\d1 \times\Omega}p_i(\bx)\tilde{\phi}(\bx, z) d\bx d\bar{\mu}(z) \Big | \nonumber\\
&\ll& \lip[\tilde{\phi}]\cdot r+r^{-d^2-1}\no[\tilde{\phi}]_{C^1}T^{-\alpha_1/(d-1)} \nonumber\\
&\ll&\lip[{\phi}]\cdot r^{1/2}+r^{-d^2-3/2}\no[{\phi}]_{C^1}T^{-\alpha_1/(d-1)}. \nonumber
\end{eqnarray}
Setting $r=r_0T^{-2\a_2}$ for some appropriate $\a_2>0$, we get ($r_0$ depends only on the dimension)
$$\Big |\int_{I^\d1} \phi(\bx, D(T)n(\bx)\Gamma) d\bx-\int_{I^\d1 \times\Omega}\phi(\bx, z) d\bx d\bar{\mu}(z) \Big | \ll \no[\phi]_{C^1} T^{-\a_2}.$$

\end{proof}

\section{\bf Translations of a Farey Sequence and Effective Equidistribution}
The Farey fractions on the torus $\tt$ are those points whose coordinates are rational numbers. We already know that the expanding horosphere $Y=\{h\Gamma: h\in H\}$ becomes equidistributed under $F^+-$translations. We are going to study the equidistribution property of Farey fractions on $Y$ in this section. We denote by $K$ the subgroup
$$K=\left\{ A\ltimes\bb:=\begin{pmatrix}
A & \bb\\
0 & 1
\end{pmatrix}: A\in G_0, \bb\in\R^\d1 \right\}\subseteq G.$$
Let $\Lambda=\{D(s)k\Gamma: s>1, k\in K \}$.  This is an embedded submanifold of $\Omega$ by (\cite{M}, Lemma 2). For any element $\lambda\in\Lambda$, there exist unique $s> 1$ and $z\in K\Gamma/\Gamma$ such that $\lambda=D(s)z$. Let $$\cD_0=\{\bx\in\R^d: 0<x_i<x_d, 0< x_d< 1\}.$$ Marklof in \cite{M} proved that under $F^+-$translation, Farey fractions $\{n(\ahat)\Gamma: \ba\in T\cD_0\cap\zhat\}$ on the closed horosphere $\{n(\bx)\Gamma: \bx\in\tt\}$ becomes equidistributed on $\Lambda$.  We will prove an effective version of this result in Theorem \ref{farey}. 
The following lemma, which describes the behavior of $F^+-$translations of the Farey fractions, is also hinted in \cite{M}.
\begin{Lem}\label{dis} For any $T>1$, the lattice points in $T\cD_0\cap\zhat$ are in one-to-one correspondence with the intersection of $\{D(T)n(\bx)\Gamma: \bx\in I^\d1\}$ with the submanifold $\Lambda$. More precisely, {\rm {(}}recall that $\ahat=(\dd[a_1, a_d],\cdots, \dd[a_\d1,a_d])^t$ {\rm{)}}
$$\left\{\ahat: \ba\in T\cD_0\cap\zhat\right\}=\left\{\bx\in I^\d1: D(T)n(\bx)\Gamma\in\Lambda\right\}.$$

\end{Lem}
\begin{proof}
\begin{itemize}
\item [(``$\subseteq$'')] In view of Theorem \ref{ink}, for every $T\cD_0\cap\zhat$, $D(a_d)n(\ahat)\Gamma\in K\Gamma$. Therefore $D(T)n(\ahat)\Gamma\in\Lambda$.  
\item [(``$\supseteq$'')] For every $\bx\in I^\d1$ with $D(T)n(\bx)\Gamma\in\Lambda$, there exists $T'>1$ such that $D(T/T')n(\bx)\Gamma\in K\Gamma/\Gamma$. For any lattice in $K\Gamma/\Gamma$, the last coordinates of its lattice points form the set $\Z$. This means that $\dd[T, T'](x_1, \cdots, x_\d1, 1)^t\in\zhat.$ Hence $\bx=\ahat$ for some $\ba\in T\cD_0\cap\zhat$.
\end{itemize}

\end{proof}

To study effective equidistribution of the Farey sequence, we need to introduce the following

\begin{Def}
Let $\pi: G\rightarrow\Omega, \pi(g)=g\Gamma$ be the natural projection. For $g\in G$ and $x\in\Omega$, we set $|g|_\infty:=\max\{|a_{ij}|: g=(a_{ij})\}$ and $|x|_\infty:=\inf\{|g|_\infty: \pi(g)=x\}$. Let $\cC\subseteq \Omega$ be a Borel subset, we define
$|\cC|:=\max\left(1, \sup\left\{|x|_\infty: x\in\cC\right\}\right).\nonumber$
\end{Def}

\begin{Rem}\label{radius}
It follows from the definition that for every $g\in G, A\in G_0, x\in\Omega$ and $\bb\in\R^{d-1}$
\begin{enumerate}
\item $|gx|_\infty\ll |g|_\infty |x|_\infty$;
\item $|(A\ltimes\bb)\Gamma|_\infty\ll |A\Gamma|_\infty$. (as $(A\ltimes\bb)\Gamma=(A\ltimes\bb')\Gamma$ for some $\no[\bb']\ll|A|_\infty$)
\item Moreover, $|\cC|$ is finite for every relatively compact subset $\cC\subset\Omega$.
\end{enumerate}
\end{Rem}

\begin{Lem}\label{injective}Let $\tilde{\pi}$ be the map $\R^{d-1}\times F^+\times K\G/\G\rightarrow\Omega$ defined by $\tilde{\pi}(\bx, D(s), z)=n(\bx)D(s)z$, and $\cC$ be a relatively compact subset of $K\Gamma/\Gamma$. Then the restriction of $\tilde{\pi}$ on $B(1/(4d|\cC|))\times F^+\times \cC$ is injective.
\end{Lem}
\begin{proof}
Let $r=1/(2d|\cC|)$. It is enough to show that if $n(\bx)D(T)k_1\Z^d=k_2\Z^d$ for some $\bx\in B(r)$, $T\geq1$, $k_1\Gamma, k_2\Gamma\in\cC$, then $\bx=\textbf{0}$, $T=1$. To prove this, we choose $k_1, k_2$ so that $|k_1|_\infty, |k_2|_\infty<2|\cC|$, and let $k_1=A\ltimes\bb$.
The last coordinates of the lattice points in $k_2\Z^d$ form the set $\Z$, so does the $\Z-$span of the entries in the last row of $n(\bx)D(T)k_1$, i.e.
$$ T^{-1/(\d1)}(\Z \bx\cdot\ba_1+ \cdots+\Z \bx\cdot\ba_\d1)+\Z (T^{-1/(\d1)}\bx\cdot\bb+T)=\Z,$$
where $\ba_i$'s are the columns of $A$. By the choice of $r$ we have $|\bx\cdot\ba_i|<1$, so $\bx=0, T=1$. 
\end{proof}

Let $dk$ be the left Haar measure on $K$ such that $dk=d{\mu}_0d\bb$, where $d\bb$ is the Lebesgue measure on $\R^\d1$, and let $d\bar{k}$ be the induced probability measure on $K\Gamma/\Gamma$. 
According to Siegel's volume formula (cf. \cite{SI} and \cite{M}) and Theorem 8.32 of \cite{BO}, for any $f\in L^1(G)$

\begin{equation}\label{siegel}
\int_{HFK}f(g)d\mu(g)=\dd[1, \zeta(d)]\int_{H\times\R^+\times K}f(hD(s)k)d\nu(h)\dd[ds, s^{d+1}]dk.
\end{equation}
This naturally defines a Borel measure on $\Lambda$: $d\lambda=s^{-(d+1)}dsd\bar{k}$. We also consider for every smooth function $\phi$ on $I^\d1\times\Lambda$ the $C^1$-norm given by $$\no[\phi]_{C^1}:=\lo[f]+\sum_{i=1}^{d-1} \|\dd[\partial,\partial x_i]f\|_{L^{\infty}}+\sum_{X}\lo[\partial X(f)], \quad X\in \cX\cap\Big({\rm Lie}(F)+{\rm Lie}(G_0)+{\rm Lie}(H^-)\Big).$$

\begin{Thm}\label{farey}
Let $\cC$ be a relatively compact, open subset of $K\Gamma/\Gamma$, and $\cC'$ be a compact subset of $\cC$. Then there exists a constant $\a_3>0$, so that for every $\vp\in C^1(I^\d1 \times\Lambda)$ with ${\rm supp}(\vp)\subseteq I^{d-1}\times F^+\cC'$, and $T>1$ we have
\begin{eqnarray}\Big | \dd[1, T^d]\sum_{\ba\in T\mathcal{D}_0\cap\zhat}\vp(\ahat, D(T)n(\ahat)\Gamma)-\dd[1, \zeta(d)]\int_{I^\d1\times\Lambda}\vp(\bx, \lambda)d\bx d\lambda
\Big |
\ll |\cC|^d\no[\vp]_{C^1}T^{-\a_3}. \nonumber
\end{eqnarray}
\end{Thm}

\begin{proof}
{\bf Step (i)} Thicken an approximation of $\vp$ to a function $\psi\in C^1(I^\d1\times\Omega)$.

Let $0<r<r_0=1/(4d|\cC|)$. In the proof of this Theorem, we temporarily set $B_H(r):=\{n(\bx): \bx\in B(r) \}$.
We choose $\h\in\Coo(\R^{d-1})$ supported in $B(r)$ according to Lemma \ref{sob}; 
and $\beta\in C^\infty(F)$ so that $0\leq \beta\leq 1, supp(\beta)\subseteq\{D(s): s> e^{r/2}\}$, $\beta=1$ on $\{D(s): s\geq e^{r}\}$, and $\no[\beta]_{C^1}\ll r^{-1}$. 
We define a function $\psi$ on the open submanifold $I^{d-1}\times{B_H(r)}F^+\cC$ of $I^\d1\times\Omega$:
$$\psi(\bx, n(\by)D(s)z)=\b(D(s))\h(\by)\vp(\bx, D(s)z), \quad \forall\ x\in I^{d-1}, \by\in B(r), s>1,z\in \cC.$$ 
The function $\psi$ is well-defined by the injectivity result proved in Lemma \ref{injective} and the fact that $r<r_0$. 
By \cite[Lemma 2]{M}, $\{D(s): s\geq s_0\}K\Gamma/\Gamma$ is a closed embedded submanifold of $\Omega$ for any $s_0>0$, it follows that $\overline{B_H(r)}\cdot\{D(s): s\geq s_0\}\cdot \cC'$ is a closed subset of $\Omega$, hence the support of $\psi$ in $I^{d-1}\times{B_H(r)}F^+\cC$ is a closed subset of $I^{d-1}\times\Omega$. 
Therefore if we extend $\psi$ to a function on $I^{d-1}\times\Omega$ by setting $\psi=0$ outside the open subest $I^{d-1}\times{B_H(r)}F^+\cC\subseteq I^{d-1}\times\Omega$, we get a smooth function which we also denote by $\psi$ with abuse of notation. 
Moreover $\no[\psi]_{C^1}\ll \no[\b]_{C^1}\no[\h]_{C^1}\no[\vp]_{C^1}\ll r^{-d}\no[\vp]_{C^1}.$ Notice that we have
\begin{equation}\label{41}
\int_{I^\d1\times\Lambda}|\vp(\bx, \lambda)-\psi(\bx, \lambda)|d\bx d\lambda\ll r\no[\vp]_\infty,
\end{equation}
On the other hand, because $\vp(\ahat, D(T)n(\ahat)\Gamma)=\psi(\ahat, D(T)n(\ahat)\Gamma)$ whenever $T>e^ra_d$, hence using the fact that $\#\{\ba\in T\mathcal{D}_0: T<e^ra_d \}\ll rT^d$, it is easy to see that
\begin{equation}\label{42}
\sum_{\ba\in T\mathcal{D}_0\cap\zhat}|\vp(\ahat, D(T)n(\ahat)\Gamma)-\psi(\ahat, D(T)n(\ahat)\Gamma)|\ll T^dr\no[\vp]_\infty.
\end{equation}

{\bf Step (ii)} Compare the average of $\psi$ over the Farey sequences and horospheres.

Let $T>1, \ba\in T\mathcal{D}_0\cap\zhat$, and set $r'=r/T^{d/(\d1)}$ where $r<r_0$ as before. Let
$$\cE_r=\{x\in I^\d1: dist(x, \partial I^\d1)>r\},\quad \cM_{T,r}=\{\ba\in T\mathcal{D}_0\cap\zhat: \ahat\in \cE_{r'}, D(T)n(\ahat)\Gamma\in F^+\cC\}.$$
By our construction $\psi(\ahat, D(T)n(\ahat)\Gamma)\neq 0$ only if  $D(T)n(\ahat)\Gamma\in F^+\cC$, hence
\begin{eqnarray}
\sum_{\ba\in T\mathcal{D}_0\cap\zhat, \ahat\in \cE_{r'}}\psi(\ahat, D(T)n(\ahat)\Gamma)
&=& \sum_{\ba\in\cM_{T,r}}\int_{B(r)}\theta(\by)\psi(\ahat, D(T)n(\ahat)\Gamma)d\by.\nonumber
\end{eqnarray}
Let us consider the subset $\cZ_{T,r}:=\bigcup_{\ba\in\cM_{T, r}}\{\ahat+\by: \by\in B(r')\}$ of $I^\d1$. Our injectivity assumption implies that the union in $\cZ_{T, r}$ is disjoint. Hence we have that
\begin{eqnarray}
&\ &\dd[1, T^d]\Big|\sum_{\ba\in T\mathcal{D}_0\cap\zhat, \ahat\in \cE_{r'}}\psi(\ahat, D(T)n(\ahat)\Gamma)-\int_{\cZ_{T, r}}\psi(\bx, D(T)n(\bx)\Gamma )d\bx\Big|\label{one} \\
&=&\dd[1, T^d]\Big|\sum_{\ba\in\cM_{T,r}}\int_{B(r')}\psi(\ahat, D(T)n(\ahat+\by)\Gamma)d\by-\sum_{\ba\in\cM_{T,r}}\int_{B(r')}\psi(\ahat+\by, D(T)n(\ahat+\by)\Gamma)d\by\Big| \nonumber\\
&\ll& \lip[\psi]\int_{B(r')}\no[\by]d\by\ll \no[\varphi]_{C^1}T^{-d}.\nonumber
\end{eqnarray}
On the other hand, we have that $\{\bx\in \cE_{2r'}: \psi(\bx, D(T)n(\bx)\Gamma)\neq 0\}\subseteq \cZ_{T, r}$. To see this, notice that for any such $\bx$, we have $n(\bx_1)D(T)n(\bx)\Gamma\in F^+\cC$ for some $\bx_1\in B(r)$, because our function $\psi$ is supported in $I^\d1\times B_H(r)F^+\cC$. By Lemma \ref{dis}, $n(\bx_1)D(T)n(\bx)\Gamma=D(T)n(\ahat)\Gamma$ for some $\ba\in T\mathcal{D}_0\cap\zhat$. As $\bx\in \cE_{2r'}$, we have that $\ahat\in \cE_{r'}$. It follows from the above discussion
\begin{equation}\label{two}
\Big|\int_{I^{d-1}-\cZ_{T, r}} \psi(\bx, D(T)n(\bx)\Gamma) d\bx\Big|\ll r'\lo[\psi].
\end{equation}
Notice that $\#\{\ba\in T\cD_0\cap\zhat : \ahat\notin\cE_{r'}\}\ll T^dr'$, and $r<r_0<1$ (since $|\cC|\geq 1$) we get that 
\begin{eqnarray}
&\ &\Big|\int_{I^{d-1}} \psi(\bx, D(T)n(\bx)\Gamma) d\bx- \dd[1, T^d]\sum_{\ahat\in T\cD_0\cap\zhat}\psi(\ahat, D(T)n(\ahat)\Gamma)  \Big|\nonumber\\
&\ll& r'\lo[\varphi]+(\ref{one})+(\ref{two})
\ll r^{(1-d)}\no[\varphi]_{C^1}T^{-d/(d-1)} .\label{add}
\end{eqnarray}
{\bf Step (iii)}  Apply the equidistribution result of expanding horospheres. \\
By Theorem \ref{horo} 
\begin{eqnarray}\label{new}
\Big |\int_{I^\d1} \psi(\bx, D(T)n(\bx)\Gamma) d\bx-\int_{I^\d1 \times\Omega}\psi(\bx, z) d\bx d\bar{\mu}(z) \Big | 
\ll r^{-d}\no[\vp]_{C^1}T^{-\a_2}. 
\end{eqnarray}
Let $\cK$ be a Borel subset of $K$ which is mapped bijectively onto $\cC$ by $\pi$. By equation (\ref{siegel})
\begin{eqnarray}
\int_{I^\d1 \times\Omega}\psi(\bx, z) d\bx d\bar{\mu}(z)&=& \int_{I^\d1\times B_H(r)\cdot F^+\cdot\cK}\psi(\bx, g\Gamma)d\bx d\mu(g) \nonumber\\
&=& \dd[1, \zeta(d)]\int_{I^\d1\times\R^\d1\times\R_{>1}\times \cK}\h(\by)\psi(\bx, D(s)k\Gamma)d\bx d\by \dd[ds, s^{d+1}] dk \nonumber\\
&=& \dd[1, \zeta(d)]\int_{I^\d1\times\Lambda}\psi(\bx, \lambda)d\bx d\lambda. \label{43}
\end{eqnarray}

By setting $r=\dd[1,2]r_0T^{-\a_3}$ for some suitable constant $\a_3>0$ and combining (\ref{41}), (\ref{42}), (\ref{add}), (\ref{new}) and (\ref{43}), we conclude that for every $T>1$
\begin{eqnarray}\Big | \dd[1, T^d]\sum_{\ba\in T\mathcal{D}_0\cap\zhat}\vp(\ahat, D(T)n(\ahat)\Gamma)-\dd[1, \zeta(d)]\int_{I^\d1\times\Lambda}\vp(\bx, \lambda)d\bx d\lambda
\Big |
\ll |\cC|^d\no[\vp]_{C^1}T^{-\a_3}. \nonumber
\end{eqnarray}

\end{proof}

Recall from Theorem \ref{ink} that for any primitive lattice point $\ba\in\zzz$, the lattice $L_{\ba}$ appearing in (\ref{sug}) which produces the Frobenius number $F(\ba)$, is given by $L_{\ba}=m(\ahat)D(a_d)n(\ahat)\Z^d\cap e_d^\perp$. Let $L'_{\ba}=m(\ahat)D(a_d)n(\ahat)\G\in K\G/\G$. Translating $L'_{\ba}$ by $D(T/a_d)$, we get
$$D(T/a_d)L'_{\ba}=m(\ahat)D(T)n(\ahat)\Gamma.$$ Moreover we have $D(T/a_d)L'_{\ba}\in\Lambda$ if and only if $\ba\in T\cD_0\cap\zhat$. The following theorem shows that under this translation, the lattices $\left\{L'_{\ba}: \ba\in T\cD_0\cap\zhat\right\}$ becomes equidistributed in $\Lambda$.
\begin{Prop}\label{twist}
Let $\cC$ and $\cC'$ be as in Theorem \ref{farey}. Then there exists a constant $\a_4>0$, so that for every $\vp\in C^1(I^\d1 \times\Lambda)$ with ${\rm supp}(\vp)\subseteq I^{d-1}\times F^+\cC'$, and $T>1$ we have
\begin{eqnarray}
\Big | \dd[1, T^d]\sum_{\ba\in T\mathcal{D}_0\cap\zhat}\vp(\ahat, m(\ahat)D(T)n(\ahat)\Gamma)-\dd[1, \zeta(d)]\int_{I^\d1\times\Lambda}\vp(\bx, \lambda)d\bx d\lambda
\Big |
\ll |\cC|^d\no[\vp]_{C^1}T^{-\a_4}. \nonumber
\end{eqnarray}
\end{Prop}
\begin{Rem}\label{base}The non-effective result can be derived from Theorem \ref{farey} via the following simple fact. Suppose $\phi: X\rightarrow Y$ is a continuous map, and $\mu_n$, $\mu$ are Borel measures on $X$ so that $\mu_n$ converge to $\mu$ in the weak* topology. Then the push-forward Borel measures (on $Y$) $\phi^*(\mu_n)$ also converge to $\phi^*(\mu)$. Here the push-forward map is given by $\mathcal{T}: I^\d1\times\Lambda\rightarrow I^\d1\times\Lambda, \mathcal{T}(\bx, \lambda)=(\bx, m(\bx)\lambda)$.
\end{Rem}
\begin{proof} Let $\mathcal{T}$ be as in Remark \ref{base}. Since $\vp\circ\mathcal{T}$ is not $C^1$ in general, we need to approximate $\vp\circ\mathcal{T}$ by $C^1$ functions to get an error term estimate. We set $\cE_r=\{x\in I^\d1: dist(x, \partial I^\d1)>r\}$. Let $\h\in C^1(I^\d1)$ be a function such that $\chi_{\cE_r}\leq \h\leq\chi_{\cE_{r/2}}$ and $\no[\h]_{C^1}\ll r^{-1}$. The function $\tilde{\vp}(\bx, \lambda)=\h(\bx)\vp(\bx, m(\bx)\lambda)$ satisfies $ {\rm supp}(\tilde{\vp})\subseteq I^{\d1}\times F^+\cC_1'$ where $\cC_1'$ is a compact subset of $\cC_1=\bigcup_{x\in\cE_{r/2}}m(\bx)^{-1}\cC$. Since the entries of each $m(\bx)^{-1}\ (\bx\in\cE_{r/2})$ are bounded by $2/r$ in absolute value, it follows from Remark \ref{radius} that $|\cC_1|\ll |\cC|r^{-1}$.
\\
\\
\textbf{Claim :} There exists $n=n(d)>0$, such that $\no[\tilde{\vp}]_{C^1}\ll r^{-n}\no[\h]_{C^1}\no[\vp]_{C^1}$.\\
\textbf{Proof of the Claim:} 
We take $(\bx, \lambda)\in \cE_{r/2}\times\Lambda$ and consider the differential $d\mathcal{T}$ of $\mathcal{T}$ at this point. We use $\dd[\partial, \partial x_i]$ as the usual Euclidean tangent vector at $\by\in I^\d1$, and let $\partial_\lambda X$ be a tangent vectors at $\lambda\in\Lambda$ (see Definition \ref{haowan} and Theorem \ref{farey}). It is easy to check that
\begin{eqnarray}
d\mathcal{T}(\dd[\partial, \partial x_i])&=&\dd[\partial, \partial x_i]+\dd[1, (d-1)x_i]\partial_{m(\bx)\lambda}E_i,\  1\leq i\leq d-1 \nonumber\\
d\mathcal{T}(\partial_\lambda X)&=&\partial_{m(\bx)\lambda}Ad(m(\bx))(X), \nonumber
\end{eqnarray}
where $E_i=diag(-1,\cdots,d-2,\cdots ,-1,0)$ with $d-2$ in the $(i,i)^{th}$ entry. Hence the norm of $d\mathcal{T}$ satisfies $\no[d\mathcal{T}]\ll r^{-n}$ for some $n=n(d)$ at every $(x, \lambda)\in\cE_{r/2}\times\Lambda$. $\Box$ \\

Note that $\vp (\ahat, m(\ahat)D(T)n(\ahat)\Gamma)=\tilde{\vp}(\ahat, D(T)n(\ahat)\Gamma)$ if $\ahat\in\cE_r$. We thus have
\begin{eqnarray}
&\ &\Big | \dd[1, T^d]\sum_{\ba\in T\mathcal{D}_0\cap\zhat}\vp(\ahat, m(\ahat)D(T)n(\ahat)\Gamma)-\dd[1, \zeta(d)]\int_{I^\d1\times\Lambda}\vp(\bx, \lambda)d\bx d\lambda
\Big | \label{oo}\\
&\leq&\Big | \dd[1, T^d]\sum_{\ba\in T\mathcal{D}_0\cap\zhat}\tilde{\vp}(\ahat, D(T)n(\ahat)\Gamma)-\dd[1, \zeta(d)]\int_{I^\d1\times\Lambda}\tilde{\vp}(\bx, \lambda)d\bx d\lambda
\Big |+ \nonumber\\
&\ &\Big | \dd[1, T^d]\sum_{\ba\in T\mathcal{D}_0\cap\zhat, \ahat\in I^\d1\backslash\cE_{r}}\Big(\vp (\ahat, m(\ahat)D(T)n(\ahat)\Gamma)-\tilde{\vp}(\ahat, D(T)n(\ahat)\Gamma)\Big)\Big |+\nonumber\\
&\ &\Big| \dd[1, \zeta(d)]\int_{I^\d1\times\Lambda}\vp(\bx, \lambda)d\bx d\lambda
- \dd[1, \zeta(d)]\int_{I^\d1\times\Lambda}\tilde{\vp}(\bx, \lambda)d\bx d\lambda
\Big|  \nonumber 
\end{eqnarray}
Notice that the Haar measure $d\bar{k}$ on $K\G/\G$ is left invariant, so
$$ \int_{I^\d1\times\Lambda}\tilde{\vp}(\bx, \lambda)d\bx d\lambda= \int_{I^\d1\times\Lambda}\h(\bx){\vp}(\bx, \lambda)d\bx d\lambda.$$
Applying Theorem \ref{farey} to the function $\tilde{\vp}$ and combining the claim above, we conclude
 \begin{eqnarray}
(\ref{oo})&\ll& |\cC|^d r^{-d}\no[\tilde{\vp}]_{C^1}T^{-\a_3}+r\lo[\vp]+r\lo[\vp]\ll|\cC|^d r^{-n-d-1}\no[\vp]_{C^1}T^{-\a_3}+r\lo[\vp]  \nonumber
\end{eqnarray}
We complete the proof by setting $r=T^{-\a_4}$ for suitable $\a_4>0$.

\end{proof}

Recall that the space $\Omega_0$ is naturally embedded into $K\Gamma/\Gamma$. The map from $K\Gamma/\Gamma$ to $\Omega_0$ given by $k\Z^d\mapsto k\Z^d\cap\be_d^{\perp}$ is smooth. (As before we identify $\R^d\cap\be_d^{\perp}$ with $\R^\d1$ in the obvious way.) We consider the product Riemannian manifold of $\cD_0$ (Euclidean metric) and $\Omega_0$
$$\cM_{\cD_0}=\left\{(\bx, y, z): (\bx, y)\in\cD_0, z\in \Omega_0\right\}.$$
We will keep using the parametrization above. The product Borel measure on $\cM_{\cD_0}$ is writen as $d_{\cD_0}=d\bx dy d\bar{\mu}_0(z)$.
We set for every smooth function $f$ on $\cM_{\cD_0}$
$$\no[f]_{C^1}:=\lo[f]+\sum_{i=1}^{d-1}\|\dd[\partial,\partial x_i]f\|_{L^{\infty}}+\|\dd[\partial,\partial y]f\|_{L^{\infty}}+\sum_{X}\lo[\partial X(f)], \quad X\in \cX\cap{\rm Lie}(G_0).$$
\begin{Cor}\label{domain}
Let $\cC$ be a relatively compact, open subset of $\Omega_0$, and $\cC'$ be a compact subset of $\cC$. Then for every $\psi\in C^1(\cM_{\cD_0})$ with ${\rm supp}(\psi)\subseteq\{(\bx,y,z)\in\cM_{\cD_0}: z\in\cC'\}$, and every $T>1$ we have (the $\a_4$ below is as in Theorem \ref{twist})
$$
\Big | \dd[1, T^d]\sum_{\ba\in T\mathcal{D}_0\cap\zhat}\psi(\dd[\ba, T], m(\ahat)D(a_d)n(\ahat)\Gamma\cap\be_d^{\perp})-\dd[1, \zeta(d)]\int_{\cM_{\cD_0}}\psi(\bx, y, z)d_{\cD_0}
\Big | \ll |\cC|^{d}\no[\psi]_{C^1}T^{-\a_4}.
$$
\end{Cor}
\begin{proof}  Let $\mathcal{Q}$ be the smooth map from $I^\d1\times\Lambda$ to $\cM_{\cD_0}$ defined by 
$$\mathcal{Q}(\bx, D(y)^{-1}z')=(y\bx, y, z'\cap\be_d^{\perp}),\qquad  (\bx,y)\in I^d, z'\in K\G/\G.$$
Let $\tilde{\psi}=\psi\circ\mathcal{Q}$ be the smooth function on $I^\d1\times\Lambda$.
We are going to show that $\tilde{\psi}\in C^1(I^\d1\times\Lambda)$.
Note that for every $A\in G_0, \bb\in \R^{d-1}$
$$(A\ltimes\bb)D(s)=D(s)(A\ltimes(s^{d/(\d1)}\bb)).$$
Thus at every $w=(\bx, D(y)^{-1}z')\in I^\d1\times\Lambda$, the directional derivatives (see Definition \ref{haowan}) satisfy that $\partial_w Z(\tilde{\psi})=\partial_{\mathcal{Q}(w)} Z(\psi)$ for every $Z\in {\rm Lie}(G_0)$, and $\partial_w Y(\tilde{\psi})= 0$ for every $Y\in {\rm Lie}(H^-)$.
Let $X={\rm diag}(1/(d-1), \cdots, 1/(d-1), -1)\in {\rm Lie}(F)$. 
We have  
$$d\mathcal{Q}(\partial_w X)=\sum_{i=1}^{d-1}x_iy\dd[\partial, \partial x_i]+y\dd[\partial, \partial y].$$ 
It follows easily the function $\tilde{\psi}\in C^1(I^\d1\times\Lambda)$  and $\no[\tilde{\psi}]_{C^1}\ll\no[\psi]_{C^1}$. Moreover, 
\begin{eqnarray}
\int_{I^\d1\times\Lambda}\tilde{\psi}(\bx, \lambda)d\bx d\lambda &=& \int_{I^\d1\times I\times K\G/\G}\psi(y\bx, y, z'\cap\be_d^\perp)y^\d1d\bx dy d\bar{k}(z') \nonumber\\
&=& \int_{\cM_{\cD_0}}\psi(\bx, y, z)d_{\cD_0} \nonumber
\end{eqnarray}
On the other hand, Mahler's criterion implies that the set $\cC_1=\left\{z\in K\G/\G: z\cap \be_d^\perp\in \cC\right\}$ is a relatively compact, open subset of $K\Gamma/\Gamma$, and $\cC_1'=\left\{z\in K\G/\G: z\cap \be_d^\perp\in \cC'\right\}$ is compact. 
Moreover $|\cC_1|\ll |\cC|$ (Remark \ref{radius}).
Since ${\rm supp}(\tilde{\psi})\subseteq I^{d-1}\times F^+\cC_1'$, by Proposition \ref{twist} 
\begin{eqnarray}\label{d0}
&\ &\Big | \dd[1, T^d]\sum_{\ba\in T\mathcal{D}_0\cap\zhat}\psi(\dd[\ba, T], m(\ahat)D(a_d)n(\ahat)\Gamma\cap\be_d^\perp)-\dd[1, \zeta(d)]\int_{\cM_{\cD_0}}\psi(\bx, y, z)d_{\cD_0}
\Big |\nonumber\\
&=&\Big | \dd[1, T^d]\sum_{\ba\in T\mathcal{D}_0\cap\zhat}\tilde{\psi}(\ahat, m(\ahat)D(T)n(\ahat)\Gamma)-\dd[1, \zeta(d)]\int_{I^\d1\times\Lambda}\tilde{\psi}(\bx, \lambda)d\bx d\lambda
\Big |\ll|\cC|^{d}\no[\psi]_{C^1}T^{-\a_4}. \nonumber  
\end{eqnarray}
\end{proof}
\begin{Rem} \label{long}The equidistribution result in Corollary \ref{domain} enables us to derive (\ref{con}). Indeed, for any $\phi\in C(\Omega_0)$ let $\phi_0$ be the function on $\cM_{\cD_0}$ defined by $\phi_0(\bx, y, z):=\chi_{\cD}(\bx, y)\phi(z)$. (Recall that $L_{\ba}=m(\ahat)D(a_d)n(\ahat)\Gamma\cap\be_d^\perp$. ) Then
\begin{equation} \label{4100}
\dd[1, T^d]\sum_{\ba\in T\mathcal{D}\cap\zhat}\phi(L_{\ba})-\dd[\rm{vol}(\cD), \zeta(d)]\int_{\Omega_0}\phi d\bar{\mu}_0=
\dd[1, T^d]\sum_{\ba\in T\mathcal{D}_0\cap\zhat}\phi_0(\dd[\ba, T], L_{\ba})-\dd[1, \zeta(d)]\int_{\cM_{\cD_0}}\phi_0 d_{\cD_0}
\end{equation}
Suppose $\cD$ has boundary of Lebesgue measure zero, we can apply Corollary \ref{domain} and a weak* convergence argument to show that the above expression tends to zero as $T\rightarrow\infty$. This completes the proof of (\ref{con}).
\end{Rem}

Our discussion suggests that to study the error of (\ref{4100}), we have to deal with the error term in the equidistribution result of Corollary \ref{domain} when indicator functions are involved. Technically we consider sets with thin boundary. Let us recall a well-known
\begin{Lem}\label{app}Let $\cD$ be a bounded open subset of $\R^d$ with thin boundary, and $m$ be the Lebesgue measure. Then for every $0<r<1$ there exist $C^1$ functions $f_1, f_2$ so that $0\leq f_1\leq \chi_{\cD}\leq f_2$, $\no[f_i-\chi_{\cD}]_{L^1(m)}\ll_{\cD}r$, and $\no[f_i]_{C^1}\ll r^{-1}$ ($i=1, 2$). 
\end{Lem}
\begin{Rem}\label{sv} 
The key fact which guarantees the approximation in Lemma \ref{app} is that, for every $0<r<1$, we have $m(\{x\in\R^d: d(x, \partial\cD)<r\})\ll_{\cD}r$, where $d$ is the Euclidean distance. In view of  \cite[Lemma 1.1]{SV}, the statement of Lemma \ref{app} remains valid when $(\R^d, m)$ is replaced by $(\Omega, \bar{\mu})$. 
\end{Rem}

\begin{Prop}\label{formain}
Let $\cC, \cC'$ be as in Corollary \ref{domain}. Suppose an open subset $\cD\subseteq\cD_0$ has thin boundary as a subset of $\R^{d}$. Then there exist $\a_5>0$ and $C_\cD>0$ so that for every non-negative function $\phi\in C^1(\Omega_0)$ with ${\rm supp}(\phi)\subseteq\cC'$ and every $T>1$, we have that
\begin{equation}\label{uniform}
\Big | \dd[1, T^d]\sum_{\ba\in T\mathcal{D}\cap\zhat}\phi(L_{\ba})-\dd[\rm{vol}(\cD), \zeta(d)]\int_{\Omega_0}\phi d\bar{\mu}_0\Big | \ll C_\cD |\cC|^{d}\no[\phi]_{C^1}T^{-\a_5}.
\end{equation}
Here $\no[\phi]_{C^1}:=\lo[\phi]+\sum_{X\in (\cX\cap{\rm Lie}(G_0))}\lo[\partial X(\phi)].$
\end{Prop}
\begin{proof} For every function $f$ on $\cD_0$ and function $\psi$ on $\Omega_0$, we denote by $f\otimes\psi$ the function on $\cM_{\cD_0}$ defined by $(f\otimes\psi)(\bx, y, z)=f(\bx, y)\psi(z)$.
Since $\cD$ has thin boundary in $\R^{d}$, for every $0<r<1$ we take $f_1, f_2$ as in Lemma \ref{app} (restricted on $\cD_0$).
By Corollary \ref{domain} and Lemma \ref{app}, for $i=1,2$
 \begin{equation}\label{410}
\Big | \dd[1, T^d]\sum_{\ba\in T\mathcal{D}_0\cap\zhat}(f_i\otimes\phi)(\dd[\ba, T], L_{\ba})-\dd[1, \zeta(d)]\int_{\cM_{\cD_0}}(f_i\otimes\phi) d_{\cD_0} \Big| 
\ll |\cC|^{d}r^{-1}\no[\phi]_{C^1}T^{-\a_4}. 
\end{equation}
Again by Lemma \ref{app}, we have that $\no[(f_1\otimes\phi)-(f_2\otimes\phi)]_{L^1(\cM_{\cD_0})}\ll_{\cD} r\lo[\phi]$. Notice that $f_1\otimes\phi\leq \chi_{\cD}\otimes\phi\leq f_2\otimes\phi$ as $\phi$ is non-negative. It follows easily that estimate (\ref{uniform}) holds.
\end{proof}





\begin{flushleft}{\bf{Proof of Theorem \ref{main}:}}\end{flushleft}

Let $Q_0$ be the covering radius function as before. 
For any fixed $R>0$ we show that $\{L\in\Omega_0: Q_0(L)<R\}$ has thin boundary as a subset of $\Omega_0$.
That is, the set $E_R=\{L\in\Omega_0: Q_0(L)=R\}$ is contained in a union of finitely many bounded connected submanifolds of $\Omega_0$ of strictly lower dimension.  
In fact the statement is ``almost'' established in \cite[ Lemma 7]{M}, and we will provide some further explanation of the results there. 
In view of Remark \ref{sv}, our Theorem \ref{main}  can be deduced with essentially the same argument as in Proposition \ref{formain} . 

Let $\Sigma_1, \cdots, \Sigma_d$ be the faces of the standard simplex $\Delta_{d-1}$. We fix a Borel $\Gamma_0$-fundamental domain $\cF_0$ in $G_0$ so that every compact subset in $\Omega_0$ corresponds to a relatively compact subset in $\cF_0$, and set $\mathcal{L}_R=\left\{A\in\cF_0: Q_0(A\Z^{d-1})=R  \right\}$. By \cite[Lemma 7]{M}, 
$$\mathcal{L}_R\subseteq\bigcup_{\bn_1,\cdots, \bn_d\in\Z^{d-1}}\{A\in \cF_0: {\rm there\ exists\ }\zeta\in\R^{d-1} {\rm\ so\ that\ }A\bn_i\cap(R\Sigma_i^\circ+\zeta)\neq\emptyset\ (i=1,\cdots, d)
\}.
$$
As $\mathcal{L}_R$ is a relatively compact subset in $G_0$ (Lemma \ref{compact}), there exists $c>0$ so that $B_{R^{d-1}}(c)$ contains a fundamental domain of $A\Z^{d-1}$ for every $A\in\mathcal{L}_R$. Hence $\mathcal{L}_R$ is a subset of 
\begin{equation}\bigcup_{\bn_1,\cdots, \bn_d\in\Z^{d-1}}\left\{A\in \cF_0: {\rm there\ exists\ }\no[\zeta]<c\ {\rm so\ that\ }A\bn_i\cap(R\Sigma_i^\circ+\zeta)\neq\emptyset
\ (i=1,\cdots, d)\right\}.\label{final}
\end{equation}
Since $\mathcal{L}_R$ is relatively compact, we have that $\no[A\bn_i]\gg_{R}\no[\bn_i]$ whenever $A\in\mathcal{L}_R$. As the set $R\Sigma_i^\circ+B_{\R^{d-1}}(c)$ is bounded, it follows that in (\ref{final}) $\mathcal{L}_R$ is contained in a finite union. To complete the proof of the theorem, it suffices to show that for every fixed $d$ integral vectors $\bn_1,\cdots, \bn_d\in\Z^{d-1}$, the set
\begin{equation}\left\{A\in \cF_0: {\rm there\ exists\ }\no[\zeta]<c\ {\rm so\ that\ }A\bn_i\cap(R\Sigma_i^\circ+\zeta)\neq\emptyset
\ (i=1,\cdots, d)\right\}\label{final1}
\end{equation}
is contained in a union of finitely many bounded connected submanifolds of strictly lower dimension. In the proof of  \cite[Lemma 7]{M}, it is shown that 
\begin{equation}(\ref{final1})\subseteq\left\{A=(a_{ij})\in G_0:  {\rm tr}(LA)=R\right\},\label{final2}
\end{equation}
where $L$ is the $(n-1)\times(n-1)$ matrix whose $i$-th column is $\bn_i-\bn_d$.
Because $\mathcal{L}_R$ is a relatively compact , there is a constant $C_R>0$, so that (\ref{final2}) can be refined to
\begin{equation}(\ref{final1})\subseteq\left\{A=(a_{ij})\in G_0: |a_{ij}|<C_R,  {\rm tr}(LA)=R\right\}.\label{final3}
\end{equation}
The set $\left\{A=(a_{ij})\in G_0: |a_{ij}|<C_R,  {\rm tr}(LA)=R\right\}$ is a semi-algebraic set, and hence standard results in real algebraic geometry (see for example \cite[2.9]{BCR}) imply that it can be written as a union of finitely many bounded connected submanifold of $G_0$ of strictly lower dimension.
As the map $\pi_0: G_0\rightarrow\Omega_0$ given by $\pi_0(g)=g\Gamma_0\ (g\in G_0)$ is a local diffeomorphism, we conclude that $\{L\in\Omega_0: Q_0(L)<R\}$ has thin boundary as a subset of $\Omega_0$.


\bibliographystyle{plain}

\bigskip

\noindent Department of Mathematics, Yale University, New Haven, CT 06520, USA. 

\smallskip\noindent
li.han@yale.edu

\end{document}